\documentclass[a4paper,12pt]{amsart}

\usepackage{float}
\usepackage{euscript,eufrak,verbatim}
\usepackage{graphicx}
\usepackage[usenames]{color}
\usepackage[colorlinks,linkcolor=red,anchorcolor=blue,citecolor=blue]{hyperref}
\usepackage{amsmath}
\usepackage{amsthm}
\usepackage[all]{xy}
\usepackage{graphicx} 
\usepackage{amssymb}

\newtheorem{thm}{Theorem}[section]
\newtheorem{prop}[thm]{Proposition}

\newtheorem{lem}[thm]{Lemma}
\newtheorem{cor}[thm]{Corollary}

\newtheorem{problem}[thm]{Problem}

\def\N{\mathbb{N}}
\def\Z{\mathbb{Z}}
\def\N{\mathbb{N}}
\def\R{\mathbb{R}}

\def\XX{\mathcal{X}}
\def\YY{\mathcal{Y}}
\def\ZZ{\mathcal{Z}}

\def\SS{\mathcal{S}}

\def\SS{\mathcal{S}}

\def\coindex{\text{\rm coind}}
\def\ind{\text{\rm ind}}
\def\coindexPER{\text{\rm coind}^{\rm Per}}

\makeatletter 
\@addtoreset{equation}{section}
\numberwithin{equation}{section}

\title{Divergent coindex sequence for dynamical systems}

\author{Ruxi Shi
}
\address
{Institute of Mathematics, Polish Academy of Sciences, ul. \'Sniadeckich 8, 00-656 Warszawa, Poland}
\email{rshi@impan.pl}
\author{Masaki Tsukamoto
}
\address
{Department of Mathematics, Kyushu University, Moto-oka 744, Nishi-ku, Fukuoka 819-0395, Japan}
\email{tsukamoto@math.kyushu-u.ac.jp}
\begin{document}



\thanks{M.T. was supported by JSPS KAKENHI 18K03275.}

	\maketitle

\begin{abstract}
When a finite group freely acts on a topological space, we can define its index and coindex.
They roughly measure the size of the given action.
We explore the interaction between this index theory and topological dynamics.
Given a fixed-point free dynamical system, the set of $p$-periodic points 
admits a natural free action of $\mathbb{Z}/p\mathbb{Z}$ for each prime number $p$.
We are interested in the growth of its index and coindex as $p\to \infty$.
Our main result shows that there exists a fixed-point free dynamical system having the divergent coindex sequence.
This solves a problem posed by \cite{tsukamoto2020markerproperty}.
\end{abstract}

\section{Introduction}

\subsection{Background on $\mathbb{Z}_p$-index}  

Let $p$ be a prime number.
When the finite group $\mathbb{Z}_p := \mathbb{Z}/p\mathbb{Z}$ freely acts on a topological space, we can define 
its index. The $\mathbb{Z}_p$-index roughly measures the size of the given $\mathbb{Z}_p$-space.
It has several astonishing applications to combinatorics \cite{matouvsek2003using}.

Tsutaya, Yoshinaga and the second-named author \cite{tsukamoto2020markerproperty} 
found an application of the $\mathbb{Z}_p$-index theory to \textit{topological dynamics}.
(One of their motivations is to solve a problem about the \textit{marker property} of dynamical systems.
This will be briefly explained in \S \ref{sec:marker property}.)
The purpose of this paper is to continue this investigation.
In particular we solve a problem posed by \cite{tsukamoto2020markerproperty}.

First we prepare terminologies of $\mathbb{Z}_p$-index, following the book of 
Matou$\mathrm{\check{s}}$ek \cite{matouvsek2003using}.
A pair $(X, T)$ is called a \textbf{$\mathbb{Z}_p$-space} if $X$ is a topological space and $T:X\to X$
is a homeomorphism with $T^p = \mathrm{id}$.
It is said to be \textbf{free} if $T^a x\neq x$ for all $1\leq a\leq p-1$ and $x\in X$.
Since $p$ is a prime number, this condition is equivalent to the condition that $Tx\neq x$ for all $x\in X$.

Let $n\geq 0$ be an integer.
A free $\mathbb{Z}_p$-space $(X, T)$ is called an \textbf{$E_n \mathbb{Z}_p$-space} if it satisfies:
\begin{itemize}
  \item $X$ is an $n$-dimensional finite simplicial complex and $T$ is a simplicial map 
 (i.e. sending each simplex to simplex affinely).
  \item $X$ is $(n-1)$-connected i.e., $\pi_k(X) = 0$ for all $0\leq k \leq n-1$.
\end{itemize}
For example, $\mathbb{Z}_p$ itself (with the natural $\mathbb{Z}_p$-action) is an $E_0 \mathbb{Z}_p$-space.
(We consider that $\mathbb{Z}_p$ is $(-1)$-connected.)
The join\footnote{Recall that for two topological spaces $X$ and $Y$, the join $X*Y$ is defined by 
$$X*Y:=[0,1]\times X\times Y/\sim$$ where the equivalence relation $\sim$ is given by
$$
(0, x, y)\sim (0,x, y')~\text{and}~(1, x, y)\sim (1,x', y),
$$
for any $x,x'\in X$ and any $y,y'\in Y$.
The equivalence class of $(t, x, y)$ is denoted by $(1-t)x\oplus ty$.
Given maps $T:X\to X$ and $S:Y\to Y$, we defined the map $T*S: X*Y\to X*Y$ by 
$$T*S\left((1-t)x\oplus ty\right) = (1-t)Tx \oplus t Sy. $$}  of the $(n+1)$ copies of $\mathbb{Z}_p$
\[  \left(\mathbb{Z}_p\right)^{*(n+1)} = \underbrace{\mathbb{Z}_p* \dots *\mathbb{Z}_p}_{\text{$(n+1)$ times}} \]
is an $E_n\mathbb{Z}_p$-space. Here $\mathbb{Z}_p$ acts on each component of $\left(\mathbb{Z}_p\right)^{*(n+1)}$
simultaneously.

An $E_n\mathbb{Z}_p$-space is not unique. 
But they are essentially unique for our purpose here \cite[Lemma 6.2.2]{matouvsek2003using}:
If $(X, T)$ and $(Y, S)$ are both $E_n \mathbb{Z}_p$-spaces then 
there are equivariant continuous maps $f:X\to Y$ and $g:Y\to X$.
 
Let $(X, T)$ be a free $\mathbb{Z}_p$-space. We define its \textbf{index} and \textbf{coindex} by 
\begin{equation*}
   \begin{split}
  \ind_p (X,T)  &:=\min\{n\ge 0: \exists~\text{an equivariant continuous}: X\to E_n\Z_p   \}, \\
   \coindex_p (X,T)   & :=\max\{n\ge 0: \exists~\text{an equivariant continuous}: E_n\Z_p \to X  \}.
   \end{split}
\end{equation*}
We set $\ind_p (X,T)=\infty$ if there is no equivariant continuous map from $X$ to $E_n\Z_p$ for any $n\ge 0$. We  use the convention that $\ind_p (X,T)=\coindex_p(X,T)=-1$ for $X=\emptyset$. 
We sometime abbreviate $\ind_p (X,T)$ (resp. $\coindex_p(X,T)$) as $\ind_pX$ (resp. $\coindex_pX$).

It is known that if there exists an equivariant continuous map from $E_m\Z_p$ to $E_n\Z_p$ then $m\le n$ (see \cite[Theorem 6.2.5]{matouvsek2003using}). 
From this,
$$
\ind_p (X,T)\ge \coindex_p(X,T).
$$
Moreover, 
$\ind_p E_n\mathbb{Z}_p = \coindex_p E_n \mathbb{Z}_p = n$.

\subsection{Background on dynamical systems}

A pair $(X, T)$ is called a \textbf{(topological) dynamical system} if $X$ is a compact metrizable space and 
$T:X\to X$ is a homeomorphism.
(Notice that here we assume the compactness of $X$. This is essential for our result.)

Let $n\geq 1$ and $(X, T)$ a dynamical system.
We define $P_n(X, T)$ as the set of $n$-periodic points of $(X, T)$:
$$
P_n(X, T) := \{x\in X: \, T^n x= x\}.
$$
We often abbreviate this as $P_n(X)$.

A dynamical system $(X, T)$ is said to be \textbf{fixed-point free} it has no fixed point, i.e. 
$P_1(X, T) = \emptyset$.
It is said to be \textbf{aperiodic} (or \textbf{free}) if it has no periodic point, i.e. $P_n(X, T) = \emptyset$
for all $n\geq 1$.  

Let $(X, T)$ be a dynamical system. For each prime number $p$, the pair 
\[  \left(P_p(X, T), T\right)  \]
is a $\mathbb{Z}_p$-space.
If $(X, T)$ is a fixed-point free dynamical system then 
$\left(P_p(X, T), T\right)$ becomes a free $\mathbb{Z}_p$-space.
The paper \cite{tsukamoto2020markerproperty} investigated its index and proved 

\begin{thm}[\cite{tsukamoto2020markerproperty}, Theorem 1.2] \label{theorem: linear growth}
Let $(X, T)$ be a fixed-point free dynamical system.
The sequence 
\[  \ind_p P_p(X), \quad (p=2,3,5,7, 11, \dots) \]
has at most linear growth in $p$. Namely there exists a positive number $C$ satisfying 
\[  \ind_p P_p(X) < C\cdot p \]
for all prime numbers $p$.
\end{thm}

So the sequence $\ind_p P_p(X)$ $(p=2,3,5,\dots)$ cannot be an arbitrary sequence.
It has a nontrivial restriction.
But, \textit{is this restriction optimal?
Is there a fixed point free dynamical system $(X, T)$ such that 
\[ \ind_p P_p(X) >C\cdot p \]
for some positive number $C$ and all sufficiently large prime numbers $p$?}
This is a difficult question because (at least for our current technology) it is hard to estimate 
$\ind_p P_p(X)$ from below.

Indeed even the following simpler question has been open.

\begin{problem}[\cite{tsukamoto2020markerproperty}, Problem 7.2]\label{prob:1}
Is there a fixed-point free dynamical system $(X, T)$ such that the sequence
$$\ind_p P_p(X), \quad (p = 2, 3, 5, 7, 11, \cdots)$$
is unbounded?
\end{problem}

The main purpose of this paper is to solve this problem affirmatively.

\subsection{Main result}

Let $\SS=\R/2\Z$. Let $\rho$ be a $\SS$-invariant metric on $\SS$ defined by
$$
\rho(x, y)=\min_{n\in \Z} |x-y-2n|.
$$
Let $\sigma$ be the (left)-shift on $\SS^\Z$.
Define a subsystem of $(\SS^\Z, \sigma)$ by 
$$
\ZZ:=\left\{(x_n)_{n\in \Z} \in \SS^\Z :
 \forall n\in \Z, ~\text{either}~\rho(x_{n-1}, x_{n})\ge \frac{1}{2} ~\text{or}~ \rho(x_{n}, x_{n+1})\ge \frac{1}{2} \right\}.
$$
Obviously,  the dynamical system $(\ZZ, \sigma)$ has no fixed points and $P_p(\ZZ, \sigma)\not=\emptyset$ for all prime numbers $p$. 

Now we state our main result.
\begin{thm}\label{main thm}
We have $$\lim_{p\to \infty} \coindex_p\, P_p(\ZZ, \sigma) =\infty,$$
where $p$ runs over prime numbers.
\end{thm}

Since we know
\[  \coindex_p P_p\left(\ZZ, \sigma\right) \leq \ind_p P_p\left(\ZZ, \sigma\right), \]
we also have 
$$ \lim_{p\to \infty} \ind_p P_p\left(\ZZ, \sigma\right) = \infty. $$
So this solves Problem \ref{prob:1} affirmatively.

We would like to remark that our proof of Theorem \ref{main thm} is \textit{noneffective}.
We cannot figure out the actual growth rate of $\coindex_p P_p(\ZZ, \sigma)$ from our proof.
This remains to be a task for a future study.
The main difficulty is that (at least for the authors) it is very hard to directly study the topology of 
$P_p(\ZZ,\sigma)$.
Our proof is indirect and uses the \textit{marker property} (see \S \ref{sec:marker property}).

Using Theorem \ref{main thm}, we can also construct some other
fixed-point free dynamical systems 
having divergent coindex sequence.
Let $\rho_N$ be a $\SS^N$-invariant metric on $\SS^N$ defined by
$$
\rho_N\left( (x_i)_{i=1}^N, (x_i)_{i=1}^N \right)=\max_{1\le i\le N} \rho(x_i, y_i).
$$
For a positive integer $N$ and $\delta>0$, we define
$$
\XX(\SS^N, 1, \delta):=\left\{(x_n)_{n\in \Z}\in (\SS^N)^\Z: \rho_N(x_n, x_{n+1})\ge \delta, ~\forall n\in \Z \right\},
$$
This notation might look a bit strange. Its meaning will become clearer in \S \ref{section: inverse limit}.
The system $\XX(\SS^N, 1, \delta)$ has no fixed point.

\begin{lem}\label{lem:embedding Z}
We have the following equivariant embeddings:
$$
\XX(\SS, 1, 1/2) \hookrightarrow (\ZZ, \sigma) \hookrightarrow \XX(\SS^2, 1, 1/2)\hookrightarrow \XX(\SS^N, 1, 1/2),
$$
for all integers $N\ge 2$.
\end{lem}
\begin{proof}
The embeddings $\XX(\SS, 1, 1/2) \hookrightarrow (\ZZ, \sigma)$ 
and $\XX(\SS^2, 1, 1/2)\hookrightarrow \XX(\SS^N, 1, 1/2)$ are canonical for $N\ge 2$. 
Define $f: (\ZZ, \sigma) \to \XX(\SS^2, 1, 1/2)$ by $(x_k)_{k\in \Z} \mapsto (x_k, x_{k+1})_{k\in \Z}$. 
It is easy to check that $f$ is an equivariant embedding.
\end{proof}
Combining Lemma \ref{lem:embedding Z} and Theorem \ref{main thm}, we get that for all integers $N\ge 2$
and $0<\delta<1/2$
$$\lim_{p\to \infty} \coindex_p{P_p\left(\XX(\SS^N, 1, \delta)\right)}=\infty,$$
where $p$ runs over prime numbers.

\section{Preliminaries}

\subsection{Properties of $\Z_p$-coindex}

Let $(X, T)$ be a dynamical system.
Following \cite{shi2021marker}, for simplifying notations,
we define the \textit{periodic coindex} of $(X,T)$ as 
$$
\coindex_p^{\rm Per}(X,T)=\coindex_p(P_p(X,T), T),
$$
for prime numbers $p$. The following lemma is essentially due to \cite[Proposition 3.1 ]{tsukamoto2020markerproperty}. See also the proof in \cite[Corollary 3.3]{shi2021marker}.
\begin{lem}\label{lem:basic property}
Let $(X, T)$ and $(Y, S)$ be fixed-point free dynamical systems. Let $p$ be a prime number. Then the following properties hold.
\begin{itemize}
    \item [(1)] If there is an equivariant continuous map $f: X\to Y$ then $\coindexPER_p(X,T)\le \coindexPER_p(Y, S)$.
    \item [(2)] The system $(X*Y, T* S)$ has no fixed points and $\coindexPER_p(X*Y, T* S)\ge \coindexPER_p(X,T)+\coindexPER_p(Y, S)+1$.
\end{itemize}
\end{lem}

\subsection{Marker property}\label{sec:marker property}

For a dynamical system $(X,T)$, it is said to satisfy the {\bf marker property} if for each positive integer $N$ there exists an open set $U\subset X$ satisfying that
$$
U\cap T^{-n}U=\emptyset~\text{for all}~0<n<N~\text{and,}~X=\bigcup_{n\in \Z} T^{n}U.
$$
For example, an extension of an aperiodic minimal system has
the marker property.
Gutman \cite[Theorem 6.1]{Gut15Jaworski} proved that 
every finite dimensional aperiodic dynamical system has the marker property.
Here a dynamical system $(X, T)$ is said to be finite dimensional if the topological dimension (a.k.a the Lebesgue covering 
dimension) of $X$ is finite.
The marker property has been intensively used in the context of mean dimension theory.

Probably the marker property seems to have nothing to do with the study of $\mathbb{Z}_p$-index.
But indeed it has.

We can easily see that if a dynamical system has the marker property then 
it is aperiodic.
It had been an open problem for several years whether the converse holds or not.
This problem was solved by \cite{tsukamoto2020markerproperty}.
They constructed an aperiodic dynamical system which does not have the marker property.
The main ingredient of their proof is the $\mathbb{Z}_p$-index theory\footnote{This was a main motivation for
\cite{tsukamoto2020markerproperty} to study the interaction between $\mathbb{Z}_p$-index theory and topological dynamics.}.
The first-named author \cite{shi2021marker} further developed the argument and proved that 
there exists a finite mean dimensional aperiodic dynamical system which does not have the marker property.

The proof of Theorem \ref{main thm} uses the method developed in \cite{shi2021marker}.

\section{Inverse limit of a family of dynamical systems} \label{section: inverse limit}
In this section, we follow \cite[Section 5]{shi2021marker} and write the results from $(\R/2\Z)^{\Z}$ to infinite products of a compact metrizable abelian group.

Let $G$ be a compact metrizable abelian group. Then there is a $G$-invariant metric $\rho$ on $G$ which is compatible with its topology (\cite{struble1974metrics}), i.e. $\rho(x+g,y+g)=\rho(x,y)$ for any $x,y,g\in G$.  Let $\sigma$ be the (left)-shift on $G^\Z$, i.e. $\sigma((x_k)_{k\in \Z} )=(x_{k+1})_{k\in \Z}$. For any positive integers $m$ and any number $\delta>0$, we define a subsystem $(\XX(G, m, \delta), \sigma)$ of $(G^\Z, \sigma)$ by
$$
\XX(G, m, \delta):=\{(x_n)_{n\in \Z}\in G^\Z: \rho(x_n, x_{n+m!})\ge \delta, ~\forall n\in \Z \},
$$
where $m!=m\cdot (m-1)\cdot \dots \cdot 2 \cdot 1$.
It is clear that $(\XX(G, m, \delta), \sigma)$ has no fixed points. We denote by $(X_m, T_m):=(\XX(G, m, \delta), \sigma)$ for convenience when $G$ is fixed. 

For $m>1$, we define an equivariant continuous map $\theta_{m,m-1}$ from $X_m$ to $G^\Z$ by
$$
(x_{k})_{k\in \Z} \mapsto \left(\sum_{i=0}^{m-1}x_{k+i (m-1)!} \right)_{k\in \Z}.
$$
A simple computation shows that
\begin{equation*}
\begin{split}
    &\rho\left(\sum_{i=0}^{m-1}x_{k+i\cdot (m-1)!} ,   \sum_{i=0}^{m-1}x_{(k+(m-1)!)+i\cdot (m-1)!} \right)\\
    =&\rho\left(\sum_{i=0}^{m-1}x_{k+i\cdot (m-1)!} ,   \sum_{i=1}^{m}x_{k+i\cdot (m-1)!}\right)\\
    =&\rho\left(x_{k} ,  x_{k+ m\cdot (m-1)!} \right)=\rho\left(x_{k} ,  x_{k+ m!} \right),~\forall k\in \Z.
\end{split}
\end{equation*}
Then we obtain that the image of $X_m$ under $\theta_{m,m-1}$ is contained in $X_{m-1}$. For $m>n$, we define $\theta_{m,n}=\theta_{m,m-1}\circ \theta_{m-1,m-2} \circ \dots \theta_{n+1,n}$ to be  an equivariant continuous map from $X_m$ to $X_n$. 


Fix $a=(a_k)_{k\in \Z}\in G^\Z.$ For $m\ge 2$, we define a map $\eta_{m-1,m}=\eta_{m-1,m}^a: X_{m-1} \to G^\Z$ by
$\eta_{m-1, m}((x_k)_{k\in \Z})=(y_k)_{k\in \Z}$ where 
\begin{equation*}
    y_k=
    \begin{cases}
    \bigbreak
    a_k~&\text{if}~0\le k\le (m-1)\cdot (m-1)!-1,\\
    \bigbreak
    x_{k-(m-1)\cdot (m-1)!}-\sum_{i=1}^{m-1}a_{k-i\cdot (m-1)!}~&\text{if}~(m-1)\cdot (m-1)!\le k\le m!-1\\
    \sum_{i=0}^{n-1} \left( x_{i\cdot m!+(m-1)!+j}-x_{i\cdot m!+j} \right) +y_j &\text{if}~k=n\cdot m!+j \\
    \bigbreak
   & ~\text{with}~n>0~\text{and}~0\le j\le m!-1,\\
    \sum_{i=n}^{-1} \left( x_{i\cdot m!+j}-x_{i\cdot m!+(m-1)!+j} \right) +y_j &\text{if}~k= n\cdot m!+j\\
    &~\text{with}~n<0~\text{and}~0\le j\le m!-1.
    \end{cases}
\end{equation*}
Obviously, the map $\eta_{m-1,m}$ is continuous. We remark that the map $\eta_{m-1,m}$ is not equivariant. We show several properties of $\eta_{m-1,m}$ in the following lemmas.
\begin{lem}\label{lem:image of eta}
For $m\ge 2$, $\eta_{m-1,m}(X_{m-1})\subset X_m$.
\end{lem}
\begin{proof}
Let $x\in X_{m-1}$ and $y=\eta_{m-1,m}(x)$. Let $k=n\cdot m!+j$ with $n\in \Z$ and $0\le j\le m!-1$. We divide the proof in the following three cases according to the value of $n$.

\vspace{5pt}

\noindent Case 1. $n=1$. We have
$$
y_k-y_{k-m!}=x_{(m-1)!+j}-x_j+y_j-y_j=x_{(m-1)!+j}-x_{j}.
$$
Since $x\in X_{m-1}$, we have
$$
\rho(y_{k}, y_{k-m!})=\rho(x_{(m-1)!+j}, x_{j})\ge \delta.
$$
\\
Case 2. $n\ge 2$. A simple computation shows that
\begin{equation*}
    \begin{split}
        &y_{k}-y_{k-m!}\\
        =&\sum_{i=0}^{n-1} \left( x_{i\cdot m!+(m-1)!+j}-x_{i\cdot m!+j} \right) - \sum_{i=0}^{n-2} \left( x_{i\cdot m!+(m-1)!+j}-x_{i\cdot m!+j} \right)\\
        =&x_{(n-1) m!+(m-1)!+j}-x_{(n-1)\cdot m!+j}=x_{k-m!+(m-1)!}-x_{k-m!}.
    \end{split}
\end{equation*}
It follows that 
$$
\rho(y_{k}, y_{k-m!})=\rho(x_{k-m!+(m-1)!}, x_{k-m!})\ge \delta.
$$
\\
Case 3. $n\le 0$. Similarly to Case 1 and Case 2, we have
\begin{equation*}
    y_{k}-y_{k-m!}=x_{k-m!+(m-1)!}-x_{k-m!},
\end{equation*}
and consequently $\rho(y_{k}, y_{k-m!})\ge\delta.$

\vspace{5pt}

To sum up, we conclude that $y\in X_m$ and $\eta_{m-1,m}(X_{m-1})\subset X_m$.
\end{proof}

By Lemma \ref{lem:image of eta}, we see that $\eta_{m-1, m}$ is the map from $X_{m-1}$ to $X_m$. Moreover, we show in the following that $\eta_{m-1, m}$ is indeed a right inverse map of $\theta_{m, m-1}$.

\begin{lem}\label{lem:identity}
$\theta_{m,m-1}\circ\eta_{m-1,m}=\text{id}, \forall m\ge 2.$
\end{lem}
\begin{proof}
Let $x\in X_{m-1}$ and $y=\eta_{m-1,m}(x)$. Then we have
$$
\theta_{m,m-1}\circ\eta_{m-1,m}(x)=\theta_{m,m-1}(y)=\left(\sum_{i=0}^{m-1}y_{k+i\cdot (m-1)!} \right)_{k\in \Z}.
$$

If $0\le k\le (m-1)!-1$, then 
\begin{equation*}
    \begin{split}
        &\sum_{i=0}^{m-1}y_{k+i\cdot (m-1)!}\\
        =&\sum_{i=0}^{m-2}a_{k+i\cdot (m-1)!}+ \left(x_{k}-\sum_{i=1}^{m-1}a_{k+(m-1-i)\cdot (m-1)!}\right)
        =x_{k}.
    \end{split}
\end{equation*}

If $k=s\cdot m!+t\cdot (m-1)!+j>(m-1)!$ for $s\ge 0$, $0\le t\le m-1$ and $0\le j\le (m-1)!-1$, then 
\begin{equation*}
    \begin{split}
      &\sum_{i=0}^{m-1}y_{k+i\cdot (m-1)!}\\
        =& \sum_{i=t}^{m-1}y_{s\cdot m!+i\cdot (m-1)!+j}+\sum_{i=0}^{t-1}y_{(s+1) m!+i \cdot (m-1)!+j}\\
        =& \sum_{i=t}^{m-1}\sum_{\ell=0}^{s-1} \left( x_{\ell \cdot m!+(i+1)\cdot (m-1)!+j}-x_{\ell\cdot m!+i\cdot (m-1)!+j} \right) \\
        &+\sum_{i=0}^{t-1}\sum_{\ell=0}^{s} \left( x_{\ell \cdot m!+(i+1)\cdot (m-1)!+j}-x_{\ell\cdot m!+i\cdot (m-1)!+j} \right) +\sum_{i=0}^{m-1}y_{i\cdot (m-1)!+j}\\
        =&\sum_{\ell=0}^{s-1}\sum_{i=0}^{m-1} \left( x_{\ell \cdot m!+(i+1)\cdot (m-1)!+j}-x_{\ell\cdot m!+i\cdot (m-1)!+j} \right) \\
        &+\sum_{i=0}^{t-1}\left( x_{s \cdot m!+(i+1)\cdot (m-1)!+j}-x_{s\cdot m!+i\cdot (m-1)!+j} \right)+ x_j\\
        =&\sum_{\ell=0}^{s-1} \left( x_{(\ell+1) \cdot m!+j}-x_{\ell\cdot m! +j} \right) 
        +\left( x_{s \cdot m!+t\cdot (m-1)!+j}-x_{s\cdot m!+j} \right)+ x_j\\
        =& x_{s\cdot m!+j}-x_j+x_{k}-x_{s\cdot m!+j} +x_j=x_k.
    \end{split}
\end{equation*}

If $k=s\cdot m!+t\cdot (m-1)!+j$ for $s<0$, $0\le t\le m-1$ and $0\le j\le (m-1)!-1$, then by the similar computation of the case where $s\ge 0$, we have $\sum_{i=0}^{m-1}y_{k+i\cdot (m-1)!}=x_k$. This completes the proof.
\end{proof}

\begin{cor}\label{cor:properties of maps}
Let $K$ be a non-negative integer. Then the following properties hold for $m>n$.
\begin{itemize}
    \item [(i)] The map 
$
\theta_{m,n}^{*(K+1)}: X_m^{*(K+1)} \to X_n^{*(K+1)}
$
is equivariant, continuous and surjective.
    \item [(ii)] The map  $\eta_{n,m}^{*(K+1)}: X_{n}^{*(K+1)} \to X_m^{*(K+1)}$ is equivariant and continuous.
    \item [(iii)] The map $\eta_{n,m}^{*(K+1)}$ is a continuous right-inverse of $\theta_{m,n}^{*(K+1)}$, i.e. $\theta_{m,n}^{*(K+1)}\circ \eta_{n,m}^{*(K+1)}={\rm id}.$
\end{itemize}
\end{cor}
\begin{proof}
By definition of joining of spaces and maps, (i), (ii) and (iii) are clear by the argument in this section. 
\end{proof}


\section{Proof of Theorem \ref{main thm}}
Let $\Z_3:=\Z/3\Z$ as before. A $\Z_3$-invariant metric $\rho$ on the finite abelian group is $\rho(x,y)=\delta_0({x-y})$ where $\delta$ is the Dirac operator, i.e. $\delta_0(a)=0$ if and only if $a=0$. For $m\ge 1$ and $0<\delta<1$, the subshift $(\Sigma_m, \sigma):=\XX(\Z_3, m, \delta)$ of the full shift $(\Z_3^\Z, \sigma)$ has the form:
$$
\Sigma_m=\{(x_n)_{n\in \Z}\in \Z_3^\Z: x_n\not= x_{n+m!}, ~\forall n\in \Z \}.
$$

\begin{lem}\label{lem:finite periodic point}
The dynamical system $(\Sigma_m, \sigma)$ has no fixed point. The set $P_p(\Sigma_m, \sigma)$ is nonempty and finite for every prime number $p>m!$.
\end{lem}
\begin{proof}
It is obvious that $P_1(\Sigma_m, \sigma)=\emptyset$. Let $p$ be a prime number with $p>m!$. Notice that 
$$P_p(\Sigma_m, \sigma)=\{(x_i)_{i\in \Z_p}\in \Z_3^{\Z_p}: x_i\not= x_{i+m!}, ~\forall i\in \Z_p \}.$$
Since $p>m!$, we see that $p$ and $m!$ are coprime.
Let $y_k=x_{k\cdot m! \mod \Z_p}$. It follows that
$$P_p(\Sigma_m, \sigma)\cong\{(y_k)_{k\in \Z_p}\in \Z_3^{\Z_p}: y_k\not=y_{k+1}, k\in \Z_p\}. $$
It is easy to check that the right-hand side set is nonempty and finite (see also \cite[Lemma 4.1]{tsukamoto2020markerproperty}). 
\end{proof}

\begin{lem}\label{lem:p>m!}
Let $K\ge 0$. Then
$\coindexPER_p(\Sigma_m^{*(K+1)}, \sigma^{*(K+1)})=K$ for all prime numbers $p>m!$.
\end{lem}
\begin{proof}
By Lemma \ref{lem:finite periodic point}, 
we get that $P_p(\Sigma_m^{*(K+1)}, \sigma^{*(K+1)})=P_p(\Sigma_m, \sigma)^{*(K+1)}$ 
which is an $E_K\Z_p$-space for any prime number $p>m!$. 
Thus we have $\coindex_p{P_p(\Sigma_m^{*(K+1)}, \sigma^{*(K+1)})}=K$. 
This completes the proof.
\end{proof}

Let $\SS=\R/2\Z$. Let $\rho$ be a $\SS$-invariant metric on $\SS$ defined by
$$
\rho(x, y)=\min_{n\in \Z} |x-y-2n|.
$$
Define
$$
\YY:=\{(x_n)_{n\in  \Z}\in  \SS^\Z: \forall n\in \Z, ~\text{either}~\rho(x_{n-1}, x_{n})=1 ~\text{or}~ \rho(x_{n}, x_{n+1})=1   \}.
$$
This system is related to the marker property by the next lemma.

\begin{lem}[\cite{tsukamoto2020markerproperty}, Lemma 5.3]\label{lem:to Y}
Let $(X,T)$ be a dynamical system having marker property. Then there
is an equivariant continuous map from $(X, T)$ to $(\YY, \sigma)$.
\end{lem}

We recall the definition of the dynamical system $\ZZ$.
It is a subsystem of $(\SS^\Z, \sigma)$ defined by 
$$
\ZZ:=\left\{(x_n)_{n\in \Z} \in \SS^\Z : 
\forall n\in \Z, ~\text{either}~\rho(x_{n-1}, x_{n})\ge \frac{1}{2} ~\text{or}~ \rho(x_{n}, x_{n+1})\ge \frac{1}{2} \right\}.
$$
The dynamical system $(\ZZ, \sigma)$ has no fixed points and 
$P_p(\ZZ, \sigma)\not=\emptyset$ for all prime numbers $p$.

The following proposition is essentially due to \cite[Lemma 7.5]{shi2021marker}.
\begin{prop}\label{prop:inverse limit}
Let $(X,T)$ be the inverse limit of a family of dynamical systems $\{(X_n, T_n) \}_{n\in \N}$ via $\tau=(\tau_{m,n})_{m,n\in \N, m>n}$ where $\tau_{m,n}: X_m\to X_n$ are equivariant continuous maps. Suppose there is a continuous right-inverse $\gamma=(\gamma_{n,m})_{n,m\in \N, m>n}$, i.e. $\gamma_{n,m}: X_n\to X_m$ are continuous maps with  $\tau_{m,n}\circ \gamma_{n,m}={\rm id}$ for $m>n$.
If there is an equivariant continuous map $f: (X,T) \to (\YY, \sigma)$, then
there exists an integer $M$ and an equivariant continuous map
 $$g: (X_M, T_M) \longrightarrow  (\ZZ, \sigma). $$ 
\end{prop}
\begin{proof}
Let $\pi_m: \XX\to X_m$ be the natural projection for $m\in \N$. Let $P_1: \SS^\Z \to \SS$ be the projection on $0$-th coordinate. 
Define $\phi=P_1\circ f: X \to \SS$. Then $f(x)=(\phi(T^nx))_{n\in \Z}$ for any $x\in X$. Notice that there exists an integer $M>0$ such that 
\begin{equation}\label{eq:3}
    \pi_M(x)=\pi_M(y) \Longrightarrow \rho(\phi(x), \phi(y))<\frac{1}{4}.
\end{equation}
For $m\ge 1$, we define a continuous map $\gamma_m: X_m \to X$ by
$$
x\mapsto (\tau_{m,1}(x), \tau_{m,2}(x), \dots, \tau_{m,m-1}(x), x, \gamma_{m,m+1}(x), \gamma_{m,m+2}(x), \dots ).
$$Define a continuous map $\varphi=\phi\circ \gamma_M: X_M \to \SS$ and an equivariant continuous map $g: X_M \to \SS^\Z$ by 
$$
x \mapsto (\varphi(T_M^n(x)))_{n\in \Z}.
$$
Since $\pi_M\circ \gamma_M={\rm id}$ and $\pi_M\circ T=T_M\circ \pi_M$, it follows from \eqref{eq:3} that 
\begin{equation}\label{eq:4}
    \rho\left(\phi(\gamma_M(T_M^n(x)) ), \phi(T^n(\gamma_M(x)) )\right)<\frac{1}{4}, \forall n\in \Z.
\end{equation}
Fix $x\in X_M$ and $n\in \Z$. By definitions of $\YY$ and $f$, there exists an $i\in \{0,1\}$ such that
\begin{equation}\label{eq:5}
    \rho(\phi(T^{n+i}(\gamma_M(x))),\phi(T^{n+i+1}(\gamma_M(x))) )=1.
\end{equation}
Combing \eqref{eq:5} with \eqref{eq:4}, we obtain that
\begin{equation*}
\begin{split}
    &\rho\left(\phi(\gamma_M(T_M^{n+i}x) ), \phi(\gamma_M(T_M^{n+i+1}x) )\right)\\
    \ge 
    &~ \rho(\phi(T^{n+i}(\gamma_M(x))),\phi(T^{n+i+1}(\gamma_M(x))) ) \\
    &\quad - \rho\left(\phi(\gamma_M(T_M^{n+i}x) ), \phi(T^{n+i}(\gamma_M(x)) )\right)\\
    &\qquad -\rho\left(\phi(\gamma_M(T_M^{n+i+1}x) ), \phi(T^{n+i+1}(\gamma_M(x)) )\right)\\
    \ge&~ 1-\frac{1}{4}-\frac{1}{4}=\frac{1}{2}.
\end{split}
\end{equation*}
Since $\varphi=\phi\circ \gamma_M$, we have that $$\rho\left(\varphi(T_M^{n+i}x), \varphi(T_M^{n+i+1}x) \right)\ge \frac{1}{2}.$$
By definition of $g$ and arbitrariness of $n$ and $x$, we conclude that the image of $X_M$ under $g$ is contained in $\ZZ$. This completes the proof. 
\end{proof}

Now we present the proof of our main result.
\begin{proof}[Proof of Theorem \ref{main thm}]
Let $K\ge 0$.
 By Corollary \ref{cor:properties of maps} (1), let $(X,T)$ be the inverse limit of the family $\{(\Sigma_n^{*(K+1)}, \sigma^{*(K+1)}) \}_{n\in \N}$ via $\theta=(\theta_{m,n}^{*(K+1)})_{m,n\in \N, m>n}$. Since for every $m\ge 1$,$$P_{m!}(\Sigma_m^{*(K+1)}, \sigma^{*(K+1)})=P_{m!}(\Sigma_m, \sigma
 )^{*(K+1)}=\emptyset,$$ 
 we see that $(X,T)$ is aperiodic.
 Since $\Sigma_m$ is $0$-dimensional, we have that $\Sigma_m^{*(K+1)}$ is at most of dimension $K$ for any $m\ge 1$ and consequently $X$ is at most of dimension $K$ (\cite[Section 6]{nagami1970dimension}). Since an aperiodic finite dimensional dynamical system has the marker property (\cite[Theorem 6.1]{Gut15Jaworski}), the dynamical system $(X, T)$ has the marker property. By Lemma \ref{lem:to Y}, there is an equivariant continuous map from $(X,T)$ to $(\YY, \sigma)$. It follows from Corollary \ref{cor:properties of maps} and Proposition \ref{prop:inverse limit} that there exists an integer $M$ and an equivariant continuous map from $(\Sigma_M^{*(K+1)}, \sigma^{*(K+1)}))$ to $(\ZZ, \sigma)$. By Lemma \ref{lem:basic property}, we have 
 $$
  \coindexPER_p(\ZZ, \sigma)\ge \coindexPER_p(\Sigma_M^{*(K+1)}, \sigma^{*(K+1)})=K,
 $$
 for all prime numbers $p>M!$. Since $K$ is chosen arbitrarily, we conclude that 
 $$
 \lim_{p\to \infty} \coindexPER_p(\ZZ, \sigma)=\infty.
 $$
\end{proof}

\bibliographystyle{alpha}
\bibliography{universal_bib}

\def\cprime{$'$} \def\cprime{$'$}
\begin{thebibliography}{TTY20}

\bibitem[Gut15]{Gut15Jaworski}
Yonatan Gutman.
\newblock Mean dimension and {J}aworski-type theorems.
\newblock {\em Proceedings of the London Mathematical Society},
  111(4):831--850, 2015.

\bibitem[Mat03]{matouvsek2003using}
Ji{\v{r}}{\'\i} Matou{\v{s}}ek.
\newblock {\em Using the Borsuk-Ulam theorem: lectures on topological methods
  in combinatorics and geometry}.
\newblock Springer Science \& Business Media, 2003.

\bibitem[NK70]{nagami1970dimension}
Kei{\^o} Nagami and Yukihiro Kodama.
\newblock {\em Dimension theory}.
\newblock Academic Press, 1970.

\bibitem[Shi21]{shi2021marker}
Ruxi Shi.
\newblock Finite mean dimension and marker property.
\newblock {\em arXiv preprint arXiv:2102.12197}, 2021.

\bibitem[Str74]{struble1974metrics}
Raimond~A Struble.
\newblock Metrics in locally compact groups.
\newblock {\em Compositio Mathematica}, 28(3):217--222, 1974.

\bibitem[TTY20]{tsukamoto2020markerproperty}
Masaki Tsukamoto, Mitsunobu Tsutaya, and Masahiko Yoshinaga.
\newblock {$G$}-index, topological dynamics and marker property.
\newblock {\em arXiv preprint arXiv:2012.15372}, 2020.

\end{thebibliography}

\end{document}